\documentclass[10pt]{article}
\usepackage[all]{xy}
\usepackage{amsfonts,amsmath,oldgerm,amssymb,amscd}
\newcommand{\ra}{\rightarrow}

\newcommand{\by}[1]{\stackrel{#1}{\ra}}

\newcommand{\ol}{\overline}		\newcommand{\wt}{\widetilde}
\newcommand{\iso}{\by \sim}

\newtheorem{theorem}{Theorem}[section]
\newtheorem{proposition}[theorem]{Proposition}
\newtheorem{lemma}[theorem]{Lemma}

\newtheorem{corollary}[theorem]{Corollary}

\newcommand{\gj}{\blacksquare}

\newcommand{\gt}{\theta}

	\newcommand{\BF}{\mbox{$\mathbb F$}}

 	\newcommand{\BR}{\mbox{$\mathbb R$}}

	\newcommand{\BZ}{\mbox{$\mathbb Z$}}

\newcommand{\mm}{\mbox{$\mathfrak m$}}	
	\newcommand{\p}{\mbox{$\mathfrak p$}}

\newcommand{\op}{\mbox{$\oplus$}}	
	
	\newcommand{\Hom}{\mbox{\rm Hom}}
  	
\newcommand{\Um}{\mbox{\rm Um}}		\newcommand{\SL}{\mbox{\rm SL}}
\newcommand{\GL}{\mbox{\rm GL}}		\newcommand{\ot}{\mbox{$\otimes$}}

\begin{document} 

\begin{center} {\bf \Large Projective modules over overrings of polynomial rings and a question of Quillen} 
\\\vspace{.2in} {\large Manoj K. Keshari
      and Swapnil A. Lokhande
\footnote{Department of Mathematics, IIT Bombay, Mumbai - 400076, India;\;
        (keshari,swapnil)@math.iitb.ac.in}}\end{center}

\begin{abstract}
Let $(R,\mm,K)$ be a regular local ring containing a field $k$ such
that either char $k=0$ or char $k=p$ and tr-deg $K/\BF_p\geq 1$. Let
$g_1,\ldots,g_t$ be regular parameters of $R$ which are linearly
independent modulo $\mm^2$. Let $A=R_{g_1\cdots g_t}
[Y_1,\ldots,Y_m,f_1(l_1)^{-1},\ldots, f_n(l_n)^{-1}]$, where
$f_i(T)\in k[T]$ and $l_i=a_{i1}Y_1+\ldots+a_{im}Y_m$ with
$(a_{i1},\ldots,a_{im})\in k^m-(0)$. Then every projective $A$-module
of rank $\geq t$ is free. Laurent polynomial case $f_i(l_i)=Y_i$ of
this result is due to Popescu.
\end{abstract}

\section{Introduction}

{\it In this paper, we will assume that rings are commutative
  Noetherian, modules are finitely generated, projective modules are
  of constant rank and $k$ will denote a field.}

 Let $R$ be a ring and $P$ a projective $R$-module. We say that $P$ is
{\it cancellative} if $P\op R^m\iso Q\op R^m$ for some projective
$R$-module $Q$ implies $P\iso Q$. For simplicity of notations, we
begin with a definition.

\begin{define}\label{d}
A ring $A=R[Y_1,\ldots,Y_m,f_1(l_1)^{-1},\ldots,f_n(l_n)^{-1}]$ is
said to be {\bf of type $R[d, m, n]$} if $R$ is a ring of dimension
$d$, $Y_1,\ldots, Y_m$ are variables over $R$, each $f_i(T)\in R[T]$ and
either each $l_i=Y_{i_j}$ for some $i_j$, or $R$ contains a field $k$
and $l_i=\sum_{j=1}^m\,a_{ij}Y_j -b_i$ with $b_i\in R$ and
$(a_{i1},\ldots, a_{im}) \in k^m-(0)$.

Let $A$ be a ring of the type $R[d, m, n]$. We say that $A$ is { \bf
of type $R[d, m, n]^{*}$} if $f_i(T)\in k[T]$ and $b_i \in k$ for all $i$.
\end{define}
\medskip

Let $A=R[Y_1,\ldots,Y_m, f_1(Y_1)^{-1},\ldots, f_n(Y_n)^{-1}]$ be a
ring of type $R[d,m,n]$ with $n\leq m$ and $l_i=Y_i$. If $P$ is a
projective $A$-module of rank $\geq$ max $\{2,d+1\}$, then
Dhorajia-Keshari (\cite{DK}, Theorem 3.12), proved that $E(A\oplus P)$
acts transitively on $\Um(A\oplus P)$ and hence $P$ is
cancellative. This result was proved by Bass \cite{Bass} in case
$n=m=0$; Plumstead \cite{P} in case $m=1$, $n=0$; Rao \cite{Rao} in
case $n=0$; Lindel \cite{L95} in case $f_i=Y_i$.  Gabber \cite{OG}
proved the following result: {\it Let $k$ be a field and $A$ a ring of
  type $k[0,m,n]$. Then every projective $A$-module is free.} We prove
the following result (\ref{p33}) which generalizes (\cite{DK}, Theorem
3.12) and is motivated by Gabber's result.

\begin{theorem}
Let $A=R[Y_1,\ldots,Y_m, f_1(l_1)^{-1},\ldots, f_n(l_n)^{-1}]$ be a ring
  of type $R[d,m,n]$ and $P$ a projective $A$-module of rank $\geq$
  max $\{2,d+1\}$. Then $E(A\op P)$ acts transitively on $\Um(A\op
  P)$.  In particular, $P$ is cancellative.
\end{theorem}

The Bass-Quillen
conjecture \cite{B,Q} says: {\it If $R$ is a regular ring, then every
  projective module over $R[X_1,\ldots,X_r]$ is extended from $R$}.
In B-Q conjecture, we may assume that $R$ is a regular local ring, due
to Quillen's local-global principal \cite{Q}: {\it For a
  ring $B$, projective module $P$ over $B[X_1,\ldots,X_r]$ is extended
  from $B$ if and only if $P_{\mm}$ is free for
  every maximal ideal $\mm$ of $B$}. We remark that Quillen's local global
principal is also true for projective modules over positive graded rings
(\cite{Sw}, Theorem 3.1), whereas it is not true for Laurent polynomial rings
(\cite{BR}, Example 2, p. 809).

Lindel \cite{L} gave an affirmative answer to B-Q conjecture
when $R$ is a {\it regular $k$-spot}, i.e. $R=R'_{\p}$, where $R'$
is some affine $k$-algebra and $\p$ is a regular prime
ideal of $R'$. Using Lindel's result, Popescu \cite{Po2} proved 
B-Q conjecture when $R$ is any regular local ring containing a field $k$.

Let $(R,\mm)$ be a regular local ring. We say that $f\in \mm$ is a
{\it regular parameter} of $R$ if $f$ is part of a minimal generating
set of $\mm$. This is equivalent to $f\in\mm-\mm^2$. Further, let
$g_1,\ldots,g_t\in \mm$ be regular parameters. Then $g_1,\ldots,g_t$
are linearly independent modulo $\mm^2$ if and only if
$g_1,\ldots,g_t$ are part of a minimal generating set of $\mm$.

Quillen \cite{Q} had asked the following question whose affirmative
answer would imply that B-Q conjecture is true: {\it Assume $(R,\mm)$
  is a regular local ring and $f\in \mm$ a regular parameter of $R$.
  Is every projective $R_f$-module free}? 

Bhatwadekar-Rao \cite{BR} answered Quillen's question 
when $R$ is a regular $k$-spot. More generally,
they proved: {\it Let $(R,\mm)$ be a regular $k$-spot
with infinite residue field and $f$ a regular parameter of $R$.  If
$B$ is one of $R$, $R(T)$ or $R_f$, then projective modules over
$B[X_1,\ldots,X_r,Y_1^{\pm 1},\ldots,Y_s^{\pm 1}]$ are free.}  

Rao \cite{R2} generalized above result as follows: {\it
  Let $(R,\mm)$ be a regular $k$-spot with infinite residue field. Let
  $g_1,\ldots,g_t$ be regular parameters of $R$ which are linearly
  independent modulo $\mm^2$.  If $A=R_{g_1\ldots
    g_t}[X_1,\ldots,X_r,Y_1^{\pm 1},\ldots,Y_s^{\pm 1}]$, then
  projective $A$-modules of rank $\geq$ min $\{t,d/2\}$ are free.}

Popescu \cite{Po1} generalized Rao's result as follows: {\it
  Let $(R,\mm,K)$ be a regular local ring containing a field $k$ such
  that either char $k=0$ or char $k=p$ and tr-deg $K/\BF_p\geq 1$. Let
  $g_1,\ldots,g_t$ be regular parameters of $R$ which are linearly
  independent modulo $\mm^2$.  If $A=R_{g_1\ldots
    g_t}[X_1,\ldots,X_r,Y_1^{\pm 1},\ldots,Y_s^{\pm 1}]$, then
  projective $A$-modules of rank $\geq t$ are free.}

We generalize Popescu's result as follows (\ref{t1}):

\begin{theorem}\label{00}
Let $(R,\mm,K)$ be a regular local ring containing a field $k$ such
that either char $k=0$ or char $k=p$ and tr-deg $K/\BF_p\geq 1$. Let
$g_1,\ldots,g_t$ be regular parameters of $R$ which are linearly
independent modulo $\mm^2$.  If $A=R_{g_1\ldots
  g_t}[Y_1,\ldots,Y_m,f_1(l_1)^{-1},\ldots, f_n(l_n)^{-1}]$ is a ring
of type $R_{g_1\ldots g_t}[d-1,m,n]^{*}$, then every projective
$A$-module of rank $\geq t$ is free.
\end{theorem}

Note that we can not expect (\ref{00}) for rings of type
$R[d,m,n]$. For example, let $R$ be either $\BR[X,Y]_{(X,Y)}$ or
$\BR[[X,Y]]$ and $A=R[Z, f(Z)^{-1}] $ a ring of type
$R[2,1,1]$, where $f(T)=T^2+X^2+Y^2$.  Then stably free $A$-module $P$ of
rank $2$ given by the kernel of the surjection $(X,Y,Z):A^3\ra A$ is
not free. This will follow from the fact that $P$ over the rings
$\BR[X,Y,Z]_{(X,Y,Z)}[f(Z)^{-1}]$ or
$\BR[[X,Y,Z]][f(Z)^{-1}]$ is not free 
(\cite{BR}, p. 808) and (\cite{N83}, p. 366).

\section{Preliminaries}

Let $A$ be a ring and $M$ an $A$-module. We say $m\in M$
is {\it unimodular} if there exist $\phi\in M^*=\Hom_A(M,A)$ such that
$\phi(m)=1$. The set of all unimodular elements of $M$ is
denoted by $\Um(M)$. 
For an ideal $J\subset A$, we denote by $E^1(A\oplus M,J)$, the
subgroup of $Aut_A(A\oplus M)$ generated by all the automorphisms
$$\Delta_{a\varphi}=\left(
\begin{matrix}
 1 & a\varphi\\
0 & id_M
\end{matrix}  \right)
~~~ \mbox{and} ~~~ \Gamma_{m}=\left(\begin{matrix}
1&0\\
m&id_M 
\end{matrix}\right)$$
with $a\in J,\varphi \in M^*$ and $m \in M$.  In particular, if
$E_{r+1}(A)$ is the group generated by elementary matrices over $A$,
then $E^1_{r+1}(A,J)$ denotes the subgroup of $E_{r+1}(A)$
generated by
$$\Delta_{{\bf a}}=\left(
\begin{matrix}
 1 & {\bf a}\\
0 & id_{F}
\end{matrix}  \right)
~~~\mbox{and} ~~~ \Gamma_{{\bf b}}=\left(\begin{matrix}
1&0\\
{\bf b}^t & id_{F} 
\end{matrix}\right),$$ 
where $F=A^r$, ${\bf a}\in JF$ and ${\bf b} \in F$. We write $E^1(A\op
M)$ for $E^1(A\op M,A)$.

By $\Um^1(A\op M,J)$, we denote the 
set of all $(a,m)\in \Um(A\op M)$ with $a\in 1+J$, 
and $\Um(A\op M,J)$ denotes the set of all $(a,m)\in \Um^1(A\op M)$
with $m \in JM$. We write $\Um_r(A,J)$ for
$\Um(A\op A^{r-1},J)$ and $\Um^{1}_{r}(A,J)$ for $\Um^1(A\op A^{r-1},J)$. 

Let $p\in M$ and $\varphi \in M^*$ be such
that $\varphi(p)=0$. Let $\varphi_p \in End(M)$ be defined as
$\varphi_p(q)=\varphi(q)p$. Then $1+\varphi_p$ is a (unipotent)
automorphism of $M$. An automorphism of $M$ of the form $1+\varphi_p$
is called a {\it transvection} of $M$ if either $p\in \Um(M)$ or
$\varphi \in \Um(M^*)$. We denote by $E(M)$, the subgroup of $Aut(M)$
generated by all transvections of $M$.

The following result is due to Bak-Basu-Rao (\cite{bbr}, Theorem
3.10). In \cite{DK}, we proved results for $E^1(A\op P)$. Due to
this result, we can interchange $E(A\op P)$ and $E^1(A\op P)$.

\begin{theorem}\label{bb}
Let $A$ be a ring and $P$ a projective $A$-module of rank $\geq
2$. Then $E^1(A\op P)=E(A\op P)$.
\end{theorem}

The following result follows from the definition.

\begin{lemma}\label{rem2.1}
Let $I\subset J$ be ideals of a ring $A$ and $P$ a projective
$A$-module. Then the natural map $E^1(A\oplus
P, J)\rightarrow E^1(\frac{A}{I}\oplus \frac{P}{IP}, \frac{J}{I})$ is
surjective.
\end{lemma}

The following result is due to Gabber (\cite{OG}, Theorem 2.1).

\begin{theorem}\label{g}
Let $k$ be a field and
$A=k[Y_1,\ldots,Y_m,f_1(l_1)^{-1},\ldots,f_n(l_n)^{-1}]$ a ring of
type $k[0,m,n]$. Then every projective $A$-module is free.
\end{theorem}

Recall that a ring $R$ is {\it essentially of finite type over a ring
  $B$} if $R$ is localization of an affine $B$-algebra $C$ at some
multiplicative closed subset of $C$. The following result is due to
Popescu (\cite{Po2}, Theorem 3.1).

\begin{theorem}\label{Popescu}
Let $R$ be a regular local ring containing a field. Then $R$ is a
filtered inductive limit of regular local rings essentially of finite
type over $\BZ$.
\end{theorem}

The following result is due to Wiemers (\cite{W}, Proposition 2.5).

\begin{proposition}\label{w1}
Let $R$ be a ring of dimension $d$ and $A=R[X_1,\ldots,X_r,Y_1^{\pm
    1},\ldots,Y_s^{\pm 1}]$. Let $c\in \{1,X_r, Y_s-1\}$. If $s\in R$
and $r\geq max \{3,d+2\}$, then $E_r^1(A,scA)$ acts transitively on
$\Um^1_r(A,scA)$.
\end{proposition}

The following result is due to Lindel (\cite{L95}, Lemma 1.1).

\begin{lemma}\label{L1.1}
Let $A$ be a ring and $P$ a projective $A$-module of rank $r$. Then
there exist $s\in A$, $p_1,\ldots,p_r\in P$ and $\phi_1,\ldots,\phi_r
\in Hom(P,A)$ such that following holds: $P_s$ is free,
$(\phi_i(p_j))=$ diagonal $(s,\ldots,s)$, $sP\subset p_1A+\ldots
+p_rA$, the image of $s$ in $A_{red}$ is a non-zerodivisor and
$(0:sA)=(0:s^2A)$.
\end{lemma}

\begin{define}{\label{ai}}
Let $R\subset S$ be rings and $h\in R$ be a non-zerodivisor in $R$ and
$S$ both. If the natural map $R/hR\ra S/hS$ is an isomorphism, then we
say $R\ra S$ is an {\it analytic isomorphism} along $h$. In this case,
we get the following fiber product diagram
\[
\xymatrix{
R \ar@{->}[r]
     \ar@{->}[d]
& S 
     \ar@{->}[d] 
\\
R_h   
\ar@{->}[r]
     &S_h .    
}
\]

In particular, if $P$ is a projective $S$-module such that $P_h$ is
free, then $P$ is extended from $R$.
\end{define}

The following result is due to Nashier (\cite{N}, Theorem 2.8). See
also (\cite{BR}, Proposition, p. 803).

\begin{proposition}\label{BN}
Let $(R,\mm)$ be a regular $k$-spot over a perfect field $k$. Let
$g\in \mm$ and $f$ be any regular parameter of $R$ with $(g,f)$ a
regular sequence. Then there exist a field $K/k$ and a regular
$K$-spot $R'$ such that

$(i)$ $R'=K[Z_1,\ldots,Z_d]_{(\phi(Z_1),Z_2,\ldots,Z_d)}$, where
$\phi(Z_1)\in K[Z_1]$ is an irreducible monic polynomial. Moreover,
we may assume $Z_d=f$.

$(ii)$ $R'\subset R$ is an analytic isomorphism along $h$ for some
$h\in gR\cap R'$.

$(iii)$ If $R/\mm$ is infinite, then $K$ is also infinite.
\end{proposition}

We state a result due to Keshari (\cite{K09}, Lemma 3.3).

\begin{lemma}\label{K3.3}
Let $A$ be a ring and $P$ a projective $A$-module of rank $r$. Choose
$s\in A$ satisfying the properties of (\ref{L1.1}). Assume that $R^r$
is cancellative, where $R=A[X]/(X^2-s^2X).$ Then every element of
$\Um^1(A\op P, s^2A)$ can be taken to $(1,0)$ by some element of
$Aut(A\op P)$.
\end{lemma}

We end this section with following result of Bhatwadekar-Rao
(\cite{BR}, Proposition 3.7).

\begin{proposition}\label{p3.7}
Let $B$ be a reduced ring of dimension $d$ and $R$ an overring of
$B[X]$ contained in its total quotient ring. Let
$A=R[X_1,\ldots,X_n,Y_1^{\pm 1},\ldots,Y_m^{\pm 1}]$.  Then $A^r$ is
cancellative for $r\geq d+1$.
\end{proposition}


\section{Cancellation over overrings of polynomial rings}

In this section, we prove our first result (\ref{p33}).  We begin with
the following:

\begin{proposition}{\label{prop3.5}}
Let $A=R[Y_1,\ldots, Y_m,
  f_1(l_1)^{-1},\ldots, f_n(l_n)^{-1}]$ be a ring of type $R[d, m,
n]$. If $s\in R$ and $r\geq {\rm max}\{3, d+2\},$ then
$E^{1}_{r}(A, sA)$ acts transitively on $\Um^{1}_{r}(A, sA)$.
\end{proposition}

\begin{proof}
By (\cite{DK}, Lemma 3.1), we may assume that $R$ is reduced. The case
$n=0$ follows from (\ref{w1}). Assume $n>0$ and use induction on
  $n$. The case each $l_j=Y_{i_j}$ is proved in (\cite{DK},
  Proposition 3.5). We will prove the other case.

Let $(a_1,\ldots,a_r)\in \Um_r^1(A,sA)$. Recall that
$l_n=a_{n1}Y_1+\ldots +a_{nm}Y_m-b_n$ with $(a_{n1},\ldots,a_{nm})\in
k^m-(0)$ and $b_n\in R$.  We can find $\theta\in E_{m}(k)$ such that
$\theta(a_{n1},\ldots,a_{nm})=(0,\ldots,0,1)$.  Replacing the
variables $(Y_1,\ldots,Y_m)$ by $\theta(Y_1,\ldots,Y_m)$, we may
assume that $l_n=Y_m-b_n$. Further replacing $Y_m$ by $Y_m+b_n$, we
may assume that $l_n=Y_m$. 

Let $S=1+f_n(Y_m)R[Y_m]$. Then
$A_S=B[Y_1,\ldots,Y_{m-1},f_1(l_1)^{-1},\ldots,f_{n-1}(l_{n-1})^{-1}]$,
where $B=R[Y_m]_{f_n(Y_m)S}$ is a ring of dimension $d$,
$l_i=\sum_{j=1}^{m-1}a_{ij}Y_j+\wt b_i$ with $\wt b_i=a_{im}Y_m-b_i
\in B$. Hence $A_S$ is of type $B[d,m-1,n-1]$. By induction on $n$, we
can find $\sigma\in E_r^1(A_S,sA_S)$ such that $\sigma
(a_1,\ldots,a_r)=(1,0,\ldots,0)$. We can find $g=1+f_n(Y_m)h(Y_m)\in
S$ and $\sigma'\in E_r^1(A_g,sA_g)$ such that $\sigma'
(a_1,\ldots,a_r)=(1,0,\ldots,0)$.  Rest of the proof is similar to
(\cite{DK}, Proposition 3.5). Hence we only give a sketch.

Let $C=R[Y_1,\ldots,Y_m,f_1(l_1)^{-1},\ldots,f_{n-1}(l_{n-1})^{-1}]$
be a ring of type $R[d,m,n-1]$.  Consider the fiber product diagram
\[
\xymatrix{
C \ar@{->}[r]
     \ar@{->}[d]
& A=C_{f_n(Y_m)} 
     \ar@{->}[d] 
\\
C_{g(Y_m)}   
\ar@{->}[r]
     &A_{g(Y_m)}=C_{g(Y_m)f_n(Y_m)}.     
}
\]
By (\cite{DK}, Lemma 3.2), there exist $\sigma_1 \in
E_{r}^1(C_{f_n},s)$ and $\sigma_2 \in \SL_{r}^1(C_g,s )$ such that
$\sigma'$ has a splitting $\sigma' =(\sigma_2)_{f_n}\circ
(\sigma_1)_{g} $.  Patching unimodular elements $\sigma_1(a_1,\ldots,a_{r}) \in
\Um^1_r(C_{f_n},s)$ and $(\sigma_2)^{-1} (1,0,\ldots,0) \in
\Um^1_r(C_g,s)$, we get $(c_1,\ldots,c_{r}) \in \Um_{r}^1(C,s)$. By
induction on $n$, there exist $\phi \in E_{r}^1(C,s)$ such that $\phi
(c_1,\ldots,c_{r}) = (1,0,\ldots,0)$. Taking projection of $\phi$ in $A$, we
get $\Phi \in E_{r}^1(A,s)$ such that $\Phi\sigma_1(a_1,\ldots,a_{r})
= (1,0,\ldots,0)$. This completes the proof.  $ \hfill \gj$
\end{proof}
\medskip

As a consequence of (\ref{prop3.5}), we get the following:

\begin{proposition}
Let $A=R[Y_1,\ldots, Y_m,
  f_1(l_1)^{-1},\ldots, f_n(l_n)^{-1}]$ be a ring of type $R[d, m, n]$. Then

$(i)$ the canonical map $\Phi_r:\GL_r(A)/E_r(A)\ra K_1(A)$ is
surjective for $r\geq {\rm max}\{2,d+1\}$.

$(ii)$ Further assume $f_i(T)\in R[T]$ is monic polynomial, $n\leq m$ and
$l_i\in k[Y_1,\ldots,Y_i]$ with $a_{ii}\not=0$ (see \ref{d}). Then for
$r\geq {\rm max} \{3,d+2\}$, any stably elementary matrix in
$\GL_r(A)$ is in $E_r(A)$. In particular, $\Phi_{d+2}$ is an
isomorphism.
\end{proposition}

\begin{proof}
The proof of $(i)$ is same as (\cite{DK}, Theorem 3.8). For $(ii)$,
let $M\in \GL_r(A)$ be a stably elementary matrix. In case $n=0$ or
each $l_i=Y_i$, the proof follows from (\cite{DK}, Theorem
3.8). Assume $n>0$ and use induction on $n$. Recall that
$l_n=a_{n1}Y_1+\ldots+a_{nn}Y_n-b_n$ with $a_{nn}\not=0$. Changing
$Y_n\mapsto a_{nn}^{-1}(Y_n-a_{n1}Y_1-\ldots -
a_{n-1,n-1}Y_{n-1})+b_n$, we may assume that $l_n=Y_n$. Let
$S=1+f_n(Y_n)R[Y_n]$ and $B=R[Y_n]_{f_nS}$. Then
$A_S=B[Y_1,\ldots,Y_{n-1},Y_{n+1},\ldots,Y_m,f_1(l_1)^{-1},\ldots,f_{n-1}(l_{n-1})^{-1}]$
is a ring of type $B[d,m-1,n-1]$ with $l_i\in k[Y_1,\ldots,Y_i]$ and
$a_{ii}\not=0$. By induction on $n$, $M_S\in E_r(A_S)$. Hence we can
choose $g\in S$ such that $M_g\in E_r(A_g)$.  The remaining proof is
same as (\cite{DK}, Theorem 3.8), hence we omit it.  $\hfill \gj$
\end{proof}
\medskip 

In the following result, (1) will follow from (\ref{g}, \ref{L1.1})
and (2) will follow from (\cite{DK}, Lemma 3.10).

\begin{lemma}{\label{lem3.9}}
Let $A=R[Y_1,\ldots, Y_m, f_1(l_1)^{-1},\ldots, f_n(l_n)^{-1}]$ be a
ring of the type $R[d, m, n]$ and $P$ a projective $A$-module of rank
$r$.  Then there exist an $s\in R$, $p_1,\ldots, p_r\in P$ and
$\phi_1,\ldots,\phi_r\in \hom(P, A)$ such that

(1) $P_s$ is free; 
$(\phi_i(p_j))$ = diagonal $(s,\ldots, s)$;  $sP\subset
p_1A+\ldots+p_rA$; the image of $s$ in $R_{red}$ is a
non-zerodivisor; and $ (0:sR)=(0:s^2R).$
 
(2) Let $(a, p)\in \Um(A\oplus P, sA)$ with
$p=c_1p_1+\ldots+c_rp_r$, where $c_i\in sA$ for all $i$. Assume
there exist $\phi\in E^{1}_{r+1}(A, s)$ such that $\phi(a,
c_1,\ldots, c_r)=(1,0,\ldots, 0)$. Then there exist $\Phi\in
E(A\oplus P)$ such that $\Phi(a, p)=(1, 0)$.
\end{lemma}

Following is the main result of this section which generalizes 
(\cite{DK}, Theorem 3.12).

\begin{theorem}{\label{p33}}
Let $A=R[Y_{1},\ldots, Y_{m},
  f_1(l_1)^{-1},\ldots, f_n(l_n)^{-1}]$ be a ring of type $R[d,
m, n]$ and $P$ a projective $A$-module of rank $r\geq {\rm max\,} \{2,
d+1\}$. Then $E(A\oplus P)$ acts transitively on
$\Um (A\oplus P)$. In particular, $P$ is cancellative.
\end{theorem}

\begin{proof}
Using (\cite{DK}, Lemma 3.1), we may assume that $R$ is reduced. If
$d=0$, then $R$ is a direct product of fields. Hence $P$ is free by
(\ref{g}) and the result follows from (\ref{prop3.5}) with
$s=1$. Assume $d>0$ and use induction on $d$.

By (\ref{lem3.9}), there exist a non-zerodivisor $s\in R$,
$p_1,\ldots, p_r\in P$ and $\phi_1,\ldots, \phi_r\in P^*$ satisfying
the properties of (\ref{lem3.9}(1)). We may assume that $s$ is not a
unit, otherwise $P$ is free and we are done by (\ref{prop3.5}). 
Rest of the proof is similar to (\cite{DK}, Theorem
3.12) with $J=R$, we only give a sketch.

Let $(a,p)\in \Um(A\op P)$. Using (\ref{rem2.1}) and induction on $d$,
we may assume that $(a,p)=(1,0)$ modulo $s^2A$. By (\ref{lem3.9}), $p
=a_1p_1+\ldots +a_rp_r$ with $a_i \in sA$ and $(a,a_1,\ldots,a_r) \in
\Um_{r+1}(A,sA)$. By (\ref{prop3.5}), there exist $\phi\in E^1(A,sA)$
such that $\phi(a,a_1,\ldots,a_r)=(1,0,\ldots,0)$. By
(\ref{lem3.9}(2)), we get $\Psi \in E(A\op P)$ such that
$\Psi(a,p)=(1,0)$. This completes the proof.  $\hfill \gj$
\end{proof}
\medskip

Following result generalizes (\ref{p3.7}).

\begin{proposition}{\label{tt}}
Let $B$ be a reduced ring of dimension $d$ containing a field $k$ and
$R$ an overring of $B[X]$ contained in its total quotient ring. Let
$A=R[Y_1,\ldots,Y_m,f_1(l_1)^{-1},\ldots,f_n(l_n)^{-1}]$ be a ring of
type $R[\dim R,m,n]^*$ with $n\leq m$, $l_i\in k[Y_1,\ldots
  Y_i]$ and $a_{ii}\not=0$. Then every projective $A$-module of rank
$r\geq d+1$ is cancellative.
\end{proposition}

\begin{proof}
If $\dim R\leq d$ or $r\geq d+2$, then result follows from
(\ref{p33}). Hence we assume dim $R=d+1$ and $r=d+1$. 

{\bf Step 1:} We first prove that $A^{d+1}$ is cancellative. When
$n=0$, we are done by (\ref{p3.7}). Assume $n>0$ and
use induction on $n$.

Recall that $l_n=a_{n1}Y_1+\ldots+a_{nn}Y_n-b_n$ with $a_{nn}\not=0$.
Changing $Y_{n} \mapsto
a_{nn}^{-1}(Y_n-a_{n1}Y_1-\ldots-a_{n,n-1}Y_{n-1})+b_n$, we can assume
that $l_n=Y_n$. Let $P$ be a stably free $A$-module of rank $d+1$. If
$S=1+f_n(Y_n)k[Y_n]$, then $\dim B[Y_n]_{f_n(Y_n)S}=d$. If $R'=R[Y_n]_{f_nS}$, then
$A_S=R'[Y_1,\ldots,Y_{n-1},Y_{n+1},\ldots,Y_m,f_1(l_1)^{-1},\ldots,
  f_{n-1}(l_{n-1})^{-1}]$ is a ring of type
$R'[d+1,m-1,n-1]^*$ with $l_i\in k[Y_1,\ldots Y_i]$ and $a_{ii}\not=0$. By
induction on $n$, $P_S$ is free. Hence we can find $g\in k[Y_n]$ such
that $P_{1+f_ng}$ is free. If
$C'=R[Y_1,\ldots,Y_{n-1},Y_{n+1},\ldots,Y_m,f_1(l_1)^{-1},\ldots,f_{n-1}(l_{n-1})^{-1}]$
and $C=C'[Y_n]$, then we have following fiber product diagram
\[
\xymatrix{
C \ar@{->}[r]
     \ar@{->}[d]
& A=C_{f_n} 
     \ar@{->}[d] 
\\
C_{1+gf_n}   
\ar@{->}[r]
     &A_{1+gf_n}=C_{f_n(1+gf_n)}.    
}
\] 
Since $P_{1+f_ng}$ is free, by (\ref{ai}), $P$ is extended from $C$,
say $P'_{f_n}=P$ for some projective $C$-module $P'$. Since $P\op
A\iso A^{d+2}$, we get $(P'\op C)_{f_n}\iso {C_{f_n}}^{d+2}$. Since
$f_n\in C'[Y_n]$ is a monic polynomial, using Suslin's monic inversion
theorem (\cite{Su1}, Theorem 1), we get $P'\op C \iso C^{d+2}.$ But
$C$ is a ring of type $R[d+1,m,n-1]^*$ with $l_i\in k[Y_1,\ldots Y_i]$
and $a_{ii}\not=0$. Hence by induction on $n$, $C^{d+1}$ is
cancellative. Therefore, $P'$ is free and so $P$ is free. This proves
that $A^r$ is cancellative.

{\bf Step 2:} We will prove the general case.  Let $P$ be a projective
$A$-module of rank $d+1$.  If $d=0$, then we may assume that $B$ is a
field. It is easy to see that $R=B[X,f(X)^{-1}]$ for some $f(X)\in
B[X]$. Hence $A$ is a ring of type $B[0,m+1,n+1]$, so $P$ is free, by
(\ref{g}).  Assume $d\geq 1$.

If $S$ is the set of non-zerodivisors of $B$, then as above,
projective modules over $S^{-1}A$ are free. Hence we can choose $s\in
S$ such that $P_s$ is free and (\ref{lem3.9} (1)) holds. Note that if
$B'=B[T]/(T^2-s^2T)$, $R'=R[T]/(T^2-s^2T)$ and
$A'=R'[Y_1,\ldots,Y_m,f_1(l_1)^{-1},\ldots,f_n(l_n)^{-1}]$, then
$(A')^{d+1}$ is cancellative, by step 1. By (\ref{K3.3}), every
element of $\Um^1(A\op P,s^2A)$ can be taken to $(1,0)$ by some
element of $Aut(A\op P)$. To complete the proof, it is enough to show
that if $(a,p)\in \Um(A\op P)$, then there exist $\sigma\in Aut(A\op
P)$ such that $\sigma(a,p)\in \Um^1(A\op P,s^2A)$.

Let ``bar'' denote reduction modulo $s^2A$. Then $\ol A=\ol
R[Y_1,\ldots,Y_m,f_1(l_1)^{-1},\ldots,f_n(l_n)^{-1}]$ and $\dim \ol
R\leq d$. By (\ref{p33}), there exist $\ol \sigma \in E(\ol A\op \ol
P)$ such that $\ol \sigma(\ol a,\ol p)=(1,0)$. Lifting $\ol \sigma$ to
an element $\sigma\in E(A\op P)$, we get $\sigma(a,p)\in \Um^1(A\op
P,s^2A)$.  This completes the proof.  $\hfill \gj$
\end{proof}

\begin{remark}
The result (\ref{tt}) is true for rings of type $R[d,m,n]$ such that
$n\leq m$, $l_i\in k[Y_1,\ldots,Y_i]$, $a_{ii}\neq 0$ and each
$f_i(T)\in R[T]$ is a monic polynomial. The proof is same as above by
taking $S=1+f_n(Y_n)R[Y_n]$ and noting that $R'$ is an overring of
$R[Y_n]$. $\hfill \gj$
\end{remark}

\section{Quillen's question and Bhatwadekar-Rao's results}

In this section we will generalize some results from \cite{BR}
regarding Quillen's question mentioned in the introduction. We begin
with the following:

\begin{lemma}\label{cor1}
Let $R$ be a UFD of dimension $1$ and
$A=R[Y_1,\ldots,Y_m,f_1(l_1)^{-1},\ldots, f_n(l_n)^{-1}]$ a ring of
type $R[1,m,n]$. Then every projective $A$-module is free.
\end{lemma}

\begin{proof}
If $n=0$, we are done by (\cite{BR}, Proposition 3.1).  Assume $n>0$
and use induction on $n$.  Let $P$ be a projective $A$-module of rank
$r$. Using same arguments as in the proof of (\ref{prop3.5}), after
changing variables $(Y_1,\ldots,Y_m)$ by $\gt(Y_1,\ldots,Y_m)$ for
some $\theta\in E_m(k)$, we may assume that $l_n=Y_m$.

Let $S=1+f_n(Y_m)R[Y_m]$ and $R'=R[Y_m]_{f_nS}$. Then $R'$ is a UFD of
dimension $1$ and
$A_S=R'[Y_1,\ldots,Y_{m-1},f_1(l_1)^{-1},\ldots,f_{n-1}(l_{n-1})^{-1}]$
is a ring of type $R'[1,m-1,n-1]$, where $l_i=a_{i1}Y_1+\ldots
+a_{i,m-1}Y_{m-1}-\wt b_i$ with $\wt b_i =b_i-a_{im}Y_m \in R'$ for
$i=1,\ldots, n-1$.  By induction on $n$, every projective $A_S$-module
is free. In particular, $P_S$ is free. Find $1+f_ng\in
S$ such that $P_{1+f_ng}$ is free.  The ring
$C=R[Y_1,\ldots,Y_m,f_1(l_1)^{-1},\ldots,f_{n-1}(l_{n-1})^{-1}]$ is of
type $R[1,m,n-1]$. Hence by induction on $n$, projective $C$-modules
are free.  Consider the following fiber product diagram
\[
\xymatrix{
C \ar@{->}[r]
     \ar@{->}[d]
& A=C_{f_n} 
     \ar@{->}[d] 
\\
C_{1+gf_n}   
\ar@{->}[r]
     &A_{1+f_ng}=C_{f_n(1+gf_n)}.     
}
\]
Since $P_{1+f_ng}$ is free,
patching projective modules $P$ and
$(C_{1+f_ng})^r$ over $C_{f_n(1+f_ng)}$, we get that $P$ is extended
from $C$ and hence $P$ is free. $\hfill
\gj$
\end{proof} 

\subsection{Infinite residue-field case}

The following result generalizes Bhatwadekar-Rao's Laurent polynomial
case (\cite{BR}, Theorem 3.2).

 \begin{proposition}{\label{p32}}
Let $R$ be a regular $k$-spot of dimension $d$ with infinite residue
field, $f$ a regular parameter of $R$ and $A=R[Y_{1},\ldots, Y_{m},
  f_1(l_1)^{-1},\ldots, f_n(l_n)^{-1}]$ a ring of type
$R[d,m,n]^{*}$. Then every projective $A_f$-module is free.
\end{proposition}

\begin{proof}
Let $P$ be a projective $A_f$-module. If $T=R-\{0\}$, then $T^{-1}P$
is free, by (\ref{g}). Find $g\in T$ such that $P_g$ is free. We may
assume that $(g,f)$ is a regular sequence in $R$.  By (\ref{BN}),
there exist an infinite field $K/k$, a regular $K$-spot
$R'=K[Z_1,\ldots,Z_d]_{(\phi(Z_1),Z_2,\ldots, Z_d)}$ such that
$R'\subset R$ is an analytic isomorphism along $h\in gR\cap R'$ and
$f=Z_d$. Therefore, $A'=R'[Y_{1},\ldots, Y_{m}, f_1(l_1)^{-1},\ldots,
  f_n(l_n)^{-1}]$ is a ring of type $R'[d,m,n]^{*}$ and $A'\subset A$
is an analytic isomorphism along $h$.  Since $P_h$ is free, by
(\ref{ai}), $P$ is extended from $A'_{Z_d}$.

Enough to show that projective $A'_{Z_d}$-modules are
free. Replace $R'$ by $R$ and $A'$ by $A$.  If $d\leq 2$, then
$R_{Z_d}$ is a UFD of dimension $\leq 1$. Hence $P$ is free, by
(\ref{cor1}, \ref{g}).
Assume $d>2$ and use induction on $d$. The proof
is similar to (\cite{BR}, Theorem 3.2), hence we only give a sketch.

Let $S$ be multiplicative set of all non-zero homogeneous polynomials
in $C=k[Z_2,\ldots,Z_d]$. Then $R_{Z_dS}$ is a localization of
$C_S[Z_1].$ We can find $h\in C_S[Z_1]$ such that $P_S$ is defined
over the ring $D=C_S[Z_1, h(Z_1)^{-1}, Y_{1},\ldots, Y_{m},
  f_1(l_1)^{-1},\ldots, f_n(l_n)^{-1}]$. Note that $C_S$ is a UFD of
dimension $\leq 1$, by (\cite{N}, Proposition 1.11). Since $D$ is of
type $C_S[1,m+1,n+1]$, by (\ref{cor1}), $P_S$ is free. Choose $F\in S$
such that $P_F$ is free. Since $K$ is infinite, by linear change of
variables, we can assume that $F$ is homogeneous and monic polynomial
in $Z_2$ with coefficients in $k[Z_3,\ldots,Z_d]$.

If $\wt R=k[Z_1,Z_3,\ldots,Z_d]_{(\phi(Z_1),Z_3,\ldots,Z_d)}$, then
$\wt R[Z_2]\subset R$ is an analytic isomorphism along $F$ (\cite{BR},
page 803).  If $\wt A=\wt R[Z_2,
  Y_1,\ldots,Y_m,f_{1}(l_1)^{-1},\ldots,f_n(l_n)^{-1}]$, then $\wt
A_{Z_d}\subset A_{Z_d}$ is an analytic isomorphism along $F$. Since
$P_F$ is free, $P$ is extended from $\wt A_{Z_d}$, by
(\ref{ai}). Observe that $\wt A_{Z_d}$ is a ring of type $\wt
R_{Z_d}[d-2,m+1,n]$. Hence by induction on $d$, projective $\wt
A_{Z_d}$-modules are free. In particular, $P$ is free.  $\hfill \gj$
\end{proof}
\medskip

Recall that $R(T)$ denote the ring $S^{-1}R[T]$, where $S$ is the
multiplicative set consisting of all monic polynomials of $R[T]$.

\begin{corollary}\label{s2}
Let $(R,\mm)$ be a regular $k$-spot of dimension $d$ with infinite
residue field. If $A=R[Y_{1},\ldots, Y_{m}, f_1(l_1)^{-1},\ldots,
  f_n(l_n)^{-1}]$ is a ring of type $R[d,m,n]^{*}$, then projective
modules over $A$ and $A\ot_R R(T)$ are free.
\end{corollary}

\begin{proof}
$(i)$ Assume $P$ is a projective $A\ot_R R(T)$-module. By (\cite{BR},
  Corollary 3.5), $R(T)=R[X]_{(\mm,X)}[1/X]$ with $X=T^{-1}$. Since
  $f=X\in R[X]_{(\mm,X)}$ is a regular parameter, we are done by
  (\ref{p32}).

$(ii)$ Assume $P$ is a projective $A$-module. Then, we are done by $(i)$,
using Suslin's monic inversion (\cite{Su1}, Theorem 1). 
$\hfill\gj$
\end{proof}
\medskip

The laurent polynomial case of the following result is due to
Popescu \cite{Po1}. 

\begin{theorem}\label{p28}
Let $(R,\mm,K)$ be a regular local ring of dimension $d$ containing a
field $k$ such that either char $k=0$ or char $k=p$ and tr-deg
$K/\BF_p\geq 1$.  Let $f$ be a regular parameter of $R$ and
$A=R[Y_{1},\ldots, Y_{m}, f_1(l_1)^{-1},\ldots, f_n(l_n)^{-1}]$ a ring
of type $R[d,m,n]^{*}$.  Then projective modules over $A,A_f$ and
$A\ot_R R(T)$ are free.
\end{theorem}

\begin{proof}
$(i)$ Assume $P$ is a projective $A_f$-module. By (\ref{Popescu}), $R$
  is a filtered inductive limit of some regular spots $(R_i)_{i\in I}$
  over $\BZ$, in particular over the prime subfield of $R$. Further,
  we may assume that $f$ is an extension of $f'\in R_j$ for some $j$
  and that $f'$ is a regular parameter of $R_j$ (see \cite{Po1}).

Choosing possibly a bigger index $j\in I$, we may assume that $P$ is
extended from $A'_{f'}$, where
$A'=R_j[Y_{1},\ldots, Y_{m}, f_1(l_1)^{-1},\ldots, f_n(l_n)^{-1}]$ is
a ring of type $R_j[d,m,n]^{*}$. Since tr-deg $K/k\geq 1$,
we can assume that the residue field of $R_j$ is infinite. By
(\ref{p32}), $P'$ and hence $P$ is free.

$(ii)$ Following the proof of (\ref{s2}), projective modules over $A$
and $A\ot_R R(T)$ are free.  $\hfill \gj$
\end{proof}


\subsection{Finite residue-field case}

The following result is an analogue of (\ref{p32}) in case residue
field of $R$ is finite and generalizes Bhatwadekar-Rao's (\cite{BR},
Theorem 3.8).

\begin{theorem}{\label{p35}}
Let $R$ be a regular $\BF_q$-spot of dimension $d$, $f$ a regular
parameter of $R$ and
$A=R[Y_1,\ldots,Y_m,f_1(l_1)^{-1},\ldots,f_n(l_n)^{-1}]$ a ring of
type $R[d,m,n]^{*}$. Then every projective $A_f$-module of rank $\geq
d-1$ is free.
\end{theorem}

\begin{proof}
As in (\ref{p32}), using (\ref{BN}), we can assume $R=K[Z_1,\ldots,
  Z_d]_{(\phi(Z_1),Z_2,\ldots, Z_d)}$ and $f=Z_d,$ where $K\supseteq
\BF_q$ may be a finite field.  Let $P$ be a projective $A_f$-module of
rank $r\geq d-1$. Note that projective $A_f$-modules are stably
free and $\dim R_f=d-1$. Hence if $r\geq d$, then $P$ is free by
(\ref{g}, \ref{p33}). Therefore, we need to prove the result in case
$r=d-1$. We use induction on $d$.

If $d\leq 2$, then $R_f$ is a UFD of dimension $1$ and we are done by
(\ref{cor1}). Assume $d>2$.  If $\wt R=K[Z_1,\ldots,
  Z_{d-1}]_{(\phi(Z_1),Z_2,\ldots, Z_{d-1})}$, then $\wt R_{Z_{d-1}}$
is of dimension $d-2$. Since $R_{Z_dZ_{d-1}}$ is a localization of $\wt
R_{Z_{d-1}}[Z_d]$, we can find $h(Z_d) \in \wt R_{Z_{d-1}}[Z_d]$ such
that $P_{Z_{d-1}}$ is defined over $C=\wt
R_{Z_{d-1}}[Z_d,h(Z_d)^{-1},Y_1,\ldots,Y_m,f_1(l_1)^{-1},\ldots,f_n(l_n)^{-1}]$.
Since $C$ is of type $\wt R_{Z_{d-1}}[d-2,m+1,n+1]$ 
by (\ref{p33}), $P_{Z_{d-1}}$ being stably free is free.

If $R'=K[Z_1,\ldots, Z_{d-2},Z_d]_{(\phi(Z_1),Z_2,\ldots,
  Z_{d-2},Z_d)}$, then $R'_{Z_d}[Z_{d-1}]\subset R_{Z_d}$ is an
analytic isomorphism along $Z_{d-1}$ (see \cite{BR}, page 803). If
$A'=R'_{Z_d}[Z_{d-1},Y_1,\ldots,Y_m,f_1(l_1)^{-1},\ldots,f_n(l_n)^{-1}]$
then $A'_{Z_d}\subset A_{Z_d}$ is also an analytic isomorphism along
$Z_{d-1}$. Using $P_{Z_{d-1}}$ is free, $P$ is extended from
$D=R'_{Z_d}[Z_{d-1},Y_1,\ldots,Y_m,f_1(l_1)^{-1},\ldots,f_n(l_n)^{-1}]$.
By induction on $d$, projective $D$-modules of rank $\geq d-1$ are
free. Hence $P$ is free.  $\hfill \gj$
\end{proof}
\medskip

The following result is an analogue of (\ref{s2}) in case residue
field of $R$ is finite and follows from (\ref{p35}) by following the
proof of (\ref{s2}).

\begin{corollary}
Let $R$ be a regular $\BF_q$-spot of dimension $d$, and $A$ a ring of
type $R[d,m,n]^{*}$. Then projective modules of rank $\geq d$ over $A$
and $A\ot_R R(T)$ are free.
\end{corollary}
 
The proof of following result is exactly same as (\cite{BR},
Proposition 4.1, Theorem 4.2) using (\ref{g}). Hence we omit the
proof.

\begin{theorem}{\label{p36}}
 Let $R=\BF_p[[Z_1,\ldots, Z_d]]$ and $f$ be a regular parameter of
 $R$. If $A$ is a ring of type $R[d,m,n]^{*}$, then projective modules
 over $A,A_f,A\ot_R R(T)$ are free.
\end{theorem}

\section{Generalization of Rao's results}

In this section, we will generalize some results from \cite{R2}.  We
begin with the following result. It's proof is exactly same as
(\cite{R2}, Theorem 2.1) by using (\ref{p33}) instead of Swan's
result, hence we omit it. The case $t\leq 1$ is (\ref{p36}).

\begin{proposition}\label{p2.1}
 Let $(R, \mm)$ be a formal (or convergent) power series ring of
 dimension $d$ over a field $k.$ Let $g_1,\ldots, g_t$
 be regular parameters of $R$ which are linearly independent modulo
 $\mm^2.$ If $A$ is a ring of type $R[d,m,n]^{*}$, then every projective
 $A_{g_1\ldots g_t}$-module of rank $\geq t-1$ is free.
\end{proposition}

\begin{lemma}{\label{p40}}
 Let $(R, \mm)$ be a regular $k$-spot of dimension $d$ and $S$ a
 multiplicative closed subset of $R$ which contains a regular parameter of
 $R.$ Let $A=R[Y_1,\ldots, Y_m,
   f_1(l_1)^{-1},\ldots,f_n(l_n)^{-1}]$ be a ring of type
 $R[d,m,n]^{*}$. Then every projective $S^{-1}A$-module of rank $\geq
 d-1$ is free.
\end{lemma}

\begin{proof}
Since $\dim S^{-1}R\leq d-1$, if rank $P>d-1$, then we are done by
(\ref{p33}). Assume that rank $P=d-1$. We will follow the notation and
proof of (\cite{R2}, Proposition 2.3). If we show that every stably
free module $P$ of rank $d-1$ over
$R'_{Z_1s}[Y_1,\ldots,Y_m,f_1(l_1)^{-1},\ldots,f_n(l_n)^{-1}]$ is
free, then remaining proof is exactly same as in \cite{R2}.

Recall that $R'=K[Z_1,\ldots,Z_d]_{(Z_1,\ldots,Z_{d-1},\phi(Z_d))}$.
If $\wt R=K[Z_1,\ldots,Z_{d-2},Z_d]_{(Z_1,\ldots,Z_{d-2},\phi(Z_d))}$,
then $R'_{Z_1s}$ is a localization of $\wt R_{Z_1s}[Z_{d-1}]$.  We can
find $f(Z_{d-1})\in \wt R_{Z_1s}[Z_{d-1}]$ such that $P$ is defined
over $C=\wt
R_{sZ_1}[Z_{d-1},f^{-1},Y_1,\ldots,Y_m,f_1(l_1)^{-1},\ldots,f_n(l_n)^{-1}]$. Since
$C$ is a ring of type $\wt R_{Z_1}[d-2,m+1,n+1]$ and $P$ is stably
free of rank $d-1$, by (\ref{p33}), $P$ is free. This completes the
proof.  $\hfill \gj$
\end{proof}

The following result is immediate from (\ref{p40}).

\begin{corollary}
Let $(R, \mm)$ be a regular $k$-spot of dimension $3$ and $f, g, h$
 regular parameters of $R$ which are linearly independent modulo
$\mm^2.$ Let $A$ be a ring of type $R[3,m,n]^{*}$. Then
projective modules over $A, A_f, A_{fg}$ and $ A_{fgh}$ are free.
\end{corollary}

\begin{lemma}\label{p2.7}
Let $k$ be an infinite field, $B=k[Z_1,\ldots, Z_d]$,
$\mm=(Z_1,\ldots, Z_{d-1}, \phi(Z_d))$ a maximal ideal and $R=B_{\mm}$. Let
$A=R[Y_1,\ldots,Y_m,f_1(l_1)^{-1},\ldots,f_n(l_n)^{-1}]$ be a ring of
type $R[d,m,n]^*$ and $h\in k[Z_1,\ldots,Z_t]$.  Then every projective
$A_h$-module of rank $\geq t$ is free.
\end{lemma}

\begin{proof}
Assume $t=d$. If $h\in \mm$, then $\dim R_h=d-1$ and the result
follows from (\ref{p33}). If $h\not\in \mm$, then $R_h=R$ and we are
done by (\ref{s2}). The proof in case $t<d$ is
similar to (\cite{R2}, Proposition 2.7), hence we only give a
sketch.

We can find $f\in B-\mm$ such that $P$ is defined over
$B[(fh)^{-1},Y_1,\ldots,Y_m,f_1(l_1)^{-1},\ldots,f_n(l_n)^{-1}]$.  If
$S=k[Z_{t+1},\ldots,Z_d]-(0)$, then $P_S$ is defined over $\wt
R[Y_1,\ldots,Y_m,f_1(l_1)^{-1},\ldots,f_n(l_n)^{-1}]$, where $\wt
R=K[Z_1,\ldots,Z_t]_{\mm_1}[(fh)^{-1}]$ with $K=k(Z_{t+1},\ldots,Z_d)$
and $\mm_1=(Z_1,\ldots,Z_t)$.  Since rank $P\geq t$ and $\dim \wt R
\leq t$, $P_S$ is free ($t=d$ case).

Proceed as in \cite{R2}, we get that if
$B'=k[Z_1,\ldots,Z_{d-1}]_{(Z_1,\ldots,Z_{d-1})}$, then $P$ is extended from
$C_h$, where $C=
B'[Z_d,Y_1,\ldots,Y_m,f_1(l_1)^{-1},\ldots,f_n(l_n)^{-1}]$. Since
$C_h$ is of type $B'_h[d',m+1,n]^*$, where $d'\leq d-1$, by induction on
$d$, $P$ is free.  $\hfill \gj$
\end{proof}
 
\begin{lemma}\label{p2.8}
Let $K$ be an infinite field and $R=K[Z_1,\ldots, Z_d]_{\mm}$, where
$\mm=(Z_1,\ldots, Z_{d-1}, \phi(Z_d))$ is a maximal ideal. Fix $q>0$
an integer such that $d\geq 2q-1$.  Let $B=R_{hg_1\ldots g_k}$, where
$h\in K[Z_1,\ldots,Z_p]$ with $1\leq p< q$ and $g_1,\ldots,g_k$ are
regular parameters of $\mm$ with $Z_1,\ldots,Z_p,g_1,\ldots,g_k$
linearly independent modulo $\mm^2$. Let
$A=B[Y_1,\ldots,Y_m,f_1(l_1)^{-1},\ldots,f_n(l_n)^{-1}]$ be a ring of type
$B[d-1,m,n]^{*}$ and $P$ a projective $A$-module of rank $\geq d-q$. Then
there exist $g\in k[Z_1,\ldots,Z_p]$ such that $P_g$ is free.
\end{lemma}

\begin{proof}
We follow the proof and notations of (\cite{R2}, Proposition 2.8) and
indicate the necessary changes. If $k=0$, then $B=R_h$ with $h\in
K[Z_1,\ldots,Z_p]$. In this case, using (\ref{p2.7}), $P$ itself is
free. Assume $k>0$ and use induction on $k$. Proceed as in \cite{R2}
using (\ref{p40}). Let $S=K[Z_1,\ldots,Z_q]-(0)$.
We only need to show that $S^{-1}P$ is free. Remaining
arguments are same as in \cite{R2}. 

Recall that $g_1=Z_{p+1},\ldots,g_{q-p}=Z_q$.  Write
$\wt R=K(Z_1,\ldots,Z_q)[Z_{q+1},\ldots,Z_d]_{\mm'}$, where
$\mm'=(Z_{q+1},\ldots,Z_{d-1},\phi(Z_d))$. Then $S^{-1}P$
is defined over $C=\wt R_{g_{q-p+1}\ldots g_k}
[Y_1,\ldots,Y_m,f_1(l_1)^{-1},\ldots,f_n(l_n)^{-1}]$.  

Assume $k< q-p+1$. Then $C=\wt
R[Y_1,\ldots,Y_m, f_1(l_1)^{-1},\ldots,f_n(l_n)^{-1}]$ and $S^{-1}P$ 
is free, by (\ref{s2}). 
If $k\geq q-p+1$, use (\ref{p40}) to
conclude that $S^{-1}P$ is free.  $\hfill \gj$
\end{proof}

\begin{theorem}\label{t2.9}
Let $(R,\mm)$ be a regular $k$-spot of dimension $d$ with infinite
residue field. Let $A=R_{g_1\ldots
  g_t}[Y_1,\ldots,Y_m,f_1(l_1)^{-1},\ldots,f_n(l_n)^{-1}]$ be a ring
of type $R_{g_1\ldots g_t}[d-1,m,n]^{*}$, where $g_1,\ldots,g_t$
are regular parameters of $R$ which are linearly independent modulo
$\mm^2$. Then every projective
$A$-module $P$ of rank $r\geq min\{t,[d/2]\}$ is free.
\end{theorem}

\begin{proof}
We will follow the proof and notations of (\cite{R2}, Theorem 2.9). If
we show that $S^{-1}P$ is free, then rest of the argument is same as
in \cite{R2}. 
Note $R'=k[Z_1,\ldots,Z_d]_{(Z_1,\ldots,Z_{d-1},\phi(Z_d))}$
and $P$ is defined over $R'_{Z_1\ldots
  Z_t}[Y_1,\ldots,Y_m,f_1(l_1)^{-1},\ldots,f_n(l_n)^{-1}]$.  Write
$S=k[Z_1,\ldots,Z_q]-(0)$ and $\wt
R=K(Z_1,\ldots,Z_q)[Z_{q+1},\ldots,Z_d]_{(Z_{q+1},\ldots,Z_{d-1},\phi(Z_d))}$.
Then $R'$ is a localization
of $\wt R$. We can find $h_1\in K[Z_1,\ldots,Z_d]$ such that $S^{-1}P$
is defined over $C=\wt R_{h_1Z_{q+1}\ldots
  Z_t}[Y_1,\ldots,Y_m,f_1(l_1)^{-1},\ldots,f_n(l_n)^{-1}]$.  Since
$\dim \wt R=d-q$ which is $q$ if $d$ is even and $q+1$ when $d$ is
odd, $S^{-1}P$ is free, by (\ref{p40}). 
$\hfill \gj$
\end{proof}
\medskip 

The following result is an analog of (\ref{p2.1}) for regular
$k$-spots in the geometric case. Recall that a local ring $(R,\mm)$ is
said to have a coefficient field if $R$ contains a subfield $K$
isomorphic to $R/\mm$. The proof is exactly same as of (\cite{R2},
Theorem 2.12) using above results. Hence we omit the proof.

\begin{theorem}\label{t2.12}
Let $(R,\mm)$ be a regular $k$-spot with infinite residue field. Let
$g_1,\ldots,g_t$ be regular parameters of $R$ which are linearly
independent modulo $\mm^2$. Assume that $R/(g_1)$ contains a
coefficient field. If $A=R_{g_1\ldots
  g_t}[Y_1,\ldots,Y_m,f_1(l_1)^{-1},\ldots,f_n(l_n)^{-1}]$ is of type
$R_{g_1\ldots g_t}[d-1,m,n]^{*}$, then every projective $A$-module $P$ of
rank $\geq t-1$ is free.
\end{theorem}

The following result generalizes Popescu's result \cite{Po1}. For
$t\leq 1$, this follows from (\ref{p28}).

\begin{theorem}\label{t1}
Let $(R,\mm,K)$ be a regular local ring of dimension $d$ containing a
field $k$ such that either char $k=0$ or char $k=p$ and tr-deg
$K/\BF_p\geq 1$.  Let $g_1,\ldots,g_t$ be regular parameters of $R$
which are linearly independent modulo $\mm^2$. Let $A=R_{g_1\ldots
  g_t}[Y_{1},\ldots, Y_{m}, f_1(l_1)^{-1},\ldots, f_n(l_n)^{-1}]$ be a
ring of type $R_{g_1\ldots g_t}[d-1,m,n]^{*}$. Then every projective
$A$-module of rank $\geq t$ is free.
\end{theorem}

\begin{proof}
We follow the proof of (\ref{p28}) and use same notations.  As in
\cite{Po1}, if $g=g_{1}\ldots g_t$, then $g$ is an extension of $g'\in
R_j$ for some $j$. Further, $g'$ is a product of regular parameters
$g_1',\ldots, g_t'$ of $(R_j,\mm_j)$ which are linearly independent
modulo $\mm_j^2$.  If $P$ is a projective $A_f$-module, then by
choosing possibly a bigger index $j\in I$, we may assume that $P$ is
an extension of a projective module $P'$ over $A'$, where
$A'=(R_j)_{g'}[Y_{1},\ldots, Y_{m}, f_1(l_1)^{-1},\ldots,
  f_n(l_n)^{-1}]$. If $\dim R'=d'$, then $A'$ is a ring of type
$(R_j)_{g'}[d'-1,m,n]^{*}$.  Now $R_j$ is a regular $k$-spot. Since
tr-deg $K/k\geq 1$, we can assume that the residue field of $R_j$ is
infinite. By (\ref{t2.9}), $P'$ and hence $P$ is free. $\hfill \gj$
\end{proof}
\medskip

The following result is immediate from (\ref{t1}).

\begin{corollary}
Let $(R,\mm,K)$ be a regular local ring of dimension $d$ containing a
field $k$ such that either char $k=0$ or char $k=p$ and tr-deg
$K/\BF_p\geq 1$.  Let $f,g$ be regular parameters of $R$ which are
linearly independent modulo $\mm^2$. If $A$ is a ring of type
$R[d,m,n]^{*}$, then every projective module over $A,A_f$ and $A_{fg}$
are free.
\end{corollary}
\medskip

{\bf Acknowledgements:} We would like to thank Professor D. Popescu
for sending us a copy of \cite{Po1} and Mrinal K. Das for bringing
Gabber's result to our notice.

{}

\end{document}